\documentclass[11pt]{amsart}
\usepackage{
 amsmath, 
 amsxtra, 
 amsthm, 
 amssymb, 
 etex, 
 mathrsfs, 
 mathtools, 
 tikz-cd, 
 bbm,
 xr,
 comment,
 enumitem,
 enumerate}
\usepackage[toc]{appendix} 
\usepackage[all]{xy}
\usepackage{hyperref}
\usepackage{url}
\usepackage{tipa}
\usepackage{dsfont}
\usepackage{ textcomp }
\usepackage{stmaryrd} 
\usepackage{amssymb}
\usepackage{mathabx} 
\usepackage{amsmath} 
\usepackage{amscd}
\usepackage{amsbsy}
\usepackage{comment, enumerate}
\usepackage{color}
\usepackage{mathtools,caption}
\usepackage{tikz-cd}
\usepackage{longtable}
\usepackage[utf8]{inputenc}
\usepackage[OT2,T1]{fontenc}
\usepackage{changepage}
\usepackage{commath,enumerate}

\tikzcdset{scale cd/.style={every label/.append style={scale=#1},
    cells={nodes={scale=#1}}}}

\DeclareMathOperator{\Gr}{Gr}

\DeclareMathOperator{\Gal}{Gal}
\DeclareMathOperator{\rank}{rank}
\DeclareMathOperator{\ord}{ord}

\DeclareMathOperator{\corank}{corank}
\DeclareMathOperator{\Sel}{Sel}
\DeclareMathOperator{\coker}{coker}

\DeclareMathOperator{\cyc}{cyc}

\DeclareMathOperator{\length}{length}
\renewcommand{\div}{\mathrm{div}}

\newtheorem{theorem}{Theorem}[section]
\newtheorem*{theorem*}{Theorem}
\newtheorem{lemma}[theorem]{Lemma}

\newtheorem{proposition}[theorem]{Proposition}
\newtheorem{corollary}[theorem]{Corollary}
\newtheorem{defn}[theorem]{Definition}

\numberwithin{equation}{section}
\newtheorem{lthm}{Theorem} 

\newtheorem{remark}[theorem]{Remark}

\newcommand{\Tac}{\mathbf{T}^\mathrm{ac}}
\newcommand\EatDot[1]{}

\newcommand{\Ep}{E[p^\infty]}

\newcommand{\Tp}{T_pE}
\newcommand{\tEpm}{\widetilde{E}^\pm}

\newcommand{\cE}{{\mathcal{E}}}
\newcommand{\cL}{{\mathcal{L}}}
\newcommand{\cM}{{\mathcal{M}}}

\newcommand{\CC}{\mathbf{C}}
\newcommand{\Cp}{\CC_p}

\newcommand{\cS}{{\mathcal{S}}}

\newcommand{\QQ}{\mathbf{Q}}
\newcommand{\ZZ}{\mathbf{Z}}
\newcommand{\Qp}{\mathbf{Q}_p}
\newcommand{\Zp}{\mathbf{Z}_p}

\newcommand{\lb}{\llbracket}
\newcommand{\rb}{\rrbracket}

\definecolor{Green}{rgb}{0.0, 0.5, 0.0}

\newcommand{\Qcyc}{\QQ_{\cyc}}
\newcommand{\Q}{\mathbf{Q}}

\newcommand{\Z}{\mathbf{Z}}

\newcommand{\p}{\mathfrak{p}}

\newcommand{\Tr}{\mathrm{Tr}}

\newcommand{\green}[1]{\textcolor{Green}{#1}}
\newcommand{\blue}[1]{\textcolor{blue}{#1}}

\newcommand{\BDP}{\mathrm{BDP}}

  \DeclareFontFamily{U}{wncy}{}
  \DeclareFontShape{U}{wncy}{m}{n}{<->wncyr10}{}
  \DeclareSymbolFont{mcy}{U}{wncy}{m}{n}
  \DeclareMathSymbol{\sha}{\mathord}{mcy}{"58}
  \DeclareMathSymbol{\zhe}{\mathord}{mcy}{"11}
\title[Sha,  anticyclotomic extensions, supersingular primes]{Asymptotic formula for Tate--Shafarevich groups of $p$-supersingular elliptic curves over anticyclotomic extensions}

\makeatletter
\let\@wraptoccontribs\wraptoccontribs
\makeatother

\usepackage{tikz,xcolor,hyperref}

\author{Antonio Lei}
\address{Antonio Lei\newline Department of Mathematics and Statistics\\University of Ottawa\\
150 Louis-Pasteur Pvt\\
Ottawa, ON\\
Canada K1N 6N5}
\email{antonio.lei@uottawa.ca}

\author[M.F.~Lim]{Meng Fai Lim}
\address[Lim]{School of Mathematics and Statistics \& Hubei Key Laboratory of Mathematical Sciences\\ Central China Normal University\\ Wuhan\\ 430079\\ P.R.China.}
\email{limmf@ccnu.edu.cn}

\author[K. Müller]{Katharina Müller}
\address[Müller]{D\'epartement de Math\'ematiques et de Statistique\\
Universit\'e Laval, Pavillion Alexandre-Vachon\\
1045 Avenue de la M\'edecine\\
Qu\'ebec, QC\\
Canada G1V 0A6}
\email{katharina.mueller.1@ulaval.ca}


\begin{document}
\begin{abstract}
Let $p\ge 5$ be a prime number and $E/\mathbf{Q}$ an elliptic curve with good supersingular reduction at $p$. Under the generalized Heegner hypothesis, we investigate the $p$-primary subgroups of the Tate--Shafarevich groups of $E$ over number fields contained inside the anticyclotomic $\mathbf{Z}_p$-extension of an imaginary quadratic field where $p$ splits.
\end{abstract}

\maketitle

\section{Introduction}
\label{S: Intro}
Throughout this article, $p\geq 5$ is a fixed prime number. In Iwasawa theory, one is interested in the asymptotic behaviour of arithmetic objects in a tower of number fields. In the seminal work of Iwasawa \cite{Iw}, he proved that if  $k$ is a number field and $k_\infty$ is a $\Zp$-extension of $k$, then there exist integers $\mu_k$, $\lambda_k$ and $\nu_k$ such that
\[
\ord_p(h_m)=\mu_k p^m+\lambda_k m+\nu_k
\]
for $m\gg0$, where $h_m$ denotes the class number of the unique sub-extension of $k_\infty/k$ of degree $p^m$.

Ideas of Iwasawa have been generalized to the setting of abelian varieties by Mazur \cite{mazur72}. In particular, if $A/k$ is an abelian variety with  good ordinary reduction at all primes above $p$ and that the $p$-primary Selmer group over $k_\infty$ is cotorsion over the Iwasawa algebra of $\Gal(k_\infty/k)$ over $\Zp$, they showed via a control theorem on the Selmer groups that there exists integers $\mu_A$, $\lambda_A$ and $\nu_A$ such that for $m\gg0$,
\begin{equation}
    \ord_p(\sha(A/k_m)[p^\infty])=\mu_A p^m+\lambda_A m+\nu_A,
\label{eq:sha-ord}
\end{equation}
where $\sha(A/k_m)$ denotes the Tate--Shafarevich group of $A$ over $k_m$, which is assumed to be finite. (Note that the formula in \cite[P.185]{mazur72} assumes that the Mordell--Weil ranks over $k_m$ are zero; this assumption can be removed thanks to the work of  Greenberg \cite[Theorem~1.10]{Greenberg}.)

The supersingular case is more delicate since the aforementioned control theorem for Selmer groups does not hold. Let $E/\QQ$ be an elliptic curve with good supersingular reduction at $p$. Note that $a_p(E)=0$ since $p\geq 5$. In \cite{kobayashi03}, Kobayashi defined the so-called plus and minus Selmer groups over the cyclotomic $\Zp$-extension $\Qcyc$ of $\QQ$ by modifying the local conditions of the classical Selmer groups at $p$. It turns out that these new Selmer groups do satisfy a control theorem, which results in the following formula. There exist integers $\mu_E^\pm$, $\lambda_E^\pm$ and $\nu_E$ such that for $m\gg0$,
\begin{align}
\ord_p(\sha(E/k_m)[p^\infty])=\sum_{k=0}^{\lfloor\frac{m-2}{2}\rfloor}p^{m-1-2k}-\left\lfloor\frac{m}{2}\right\rfloor+\left\lfloor\frac{m}{2}\right\rfloor\lambda_E^++\left\lfloor\frac{m+1}{2}\right\rfloor\lambda_E^-\notag\\
-mr_\infty+\sum_{k=1}^{\lfloor\frac{m}{2}\rfloor}\phi(p^{2k})\mu_E^++\sum_{k=1}^{\lfloor\frac{m+1}{2}\rfloor}\phi(p^{2k-1})\mu_E^-+\nu_E,\label{eq:kob}
\end{align}
where $\phi$ is the Euler totient function and $r_\infty$ is given by $\displaystyle\lim_{m\rightarrow\infty}\rank E(k_m)$, which is known to be finite thanks to the groundbreaking work of Kato \cite{Kato} on Euler systems. Here, $k_m$ is the sub-extension of $\Qcyc/\QQ$ of degree $p^m$. Kobayashi's result has been generalized to the case $a_p(E)\ne 0$ (but still assuming $E$ has good supersingular reduction at $p$) by Sprung \cite{sprung13}.

In the present article, we study how Tate--Shafarevich groups behave over anticyclotomic $\Zp$-extensions. Let $E/\Q$ be an elliptic curve of conductor $N$ and $p\ge 5$ is a prime of supersingular reduction (in particular, the Weil bound implies that $a_p(E)=0$). We write $N=MD$ where $D$ is a square-free product of an even number of primes. Let $K$ be an imaginary quadratic field.  Throughout this article, we assume that all primes dividing $pM$ split in $K$, whereas those dividing $D$ are inert in $K$ (this is commonly known as the generalized Heegner hypothesis). This setting is often referred as the "indefinite case" (as opposed to the "definite case" where $D$ is the product of an odd number of primes). As we shall illustrate below, arithmetic objects associated to an elliptic curve in this setting often have different behaviours from the cyclotomic setting. The main divergence originates from the sign of the functional equation of $E$ over an anticyclotomic extension of $K$ being $-1$, which forces the Mordell--Weil ranks to be unbounded.

Let $K_\infty$ be the anticyclotomic $\Z_p$-extension of $K$. Let $K_n$ be the unique sub-extension of $K_\infty/K$ of degree $p^n$. Let $\Gamma$ denote the Galois group of $K_\infty/K$ and write $\Lambda$ for the Iwasawa algebra $\Zp\lb \Gamma\rb$. Throughout, we assume that both primes above $p$ are totally ramified in $K_\infty/K$ (which is equivalent to $p$ not dividing the class number of $K$).

In \cite{ciperiani}, Çiperiani proved that  if $D=1$,  the Pontryagin dual of the $p$-primary part of $\sha(E/K_\infty)$ is a torsion $\Lambda$-module. We shall extend this result to our current setting where $D$ is not necessarily 1 (see Corollary \ref{Sha tor}). Since there is no known control theorem for Tate--Shafarevich groups, the cotorsionness of the $p$-primary part of $\sha(E/K_\infty)$ does not directly give an asymptotic formula for the size of  $\sha(E/K_m)[p^\infty]$ as $m\rightarrow\infty$. Nonetheless, this cotorsionness property is in stark constrast to the cyclotomic situation, where the $p$-primary part of $\sha(E/\Qcyc)$ is of corank one over the cyclotomic Iwasawa algebra (which is a consequence of $\rank_{\ZZ}E(\Qcyc)$ being finite  and \cite[Propoisition~6.3]{lei11compositio}). Moreover, when $D=1$, Matar showed in \cite{matar21} that if the Heegner point on $E(K)$ is not $p$-divisible and the Tamagawa numbers of $E$ is not divisible by $p$, then $\sha(E/K_m)[p^\infty]$ is trivial for all $m\ge0$. Note that the context of Matar's work may be considered as the "most basic case" in Iwasawa theory, similar to the one studied by Kurihara and Pollack in the cyclotomic setting (see \cite{kurihara02,pollack05}). These observations suggest that the size of the Tate--Shafarevich groups over $K_m$ may potentially grow slower than the cyclotomic counterparts. 


We now recall that for the cyclotomic $\Zp$-extension of $\QQ$, the "most basic case" in \eqref{eq:kob} gives a formula of the form 
\begin{equation}
    \sum_{k=0}^{\lfloor\frac{m-2}{2}\rfloor}p^{m-1-2k}-\left\lfloor\frac{m}{2}\right\rfloor\label{eq:kur}.
\end{equation}
See also the work of Iovita--Pollack in \cite[\S5]{iovitapollack06}, where the size of the Tate--Shafarevich groups over a $\Zp$-extension that arises from a Lubin--Tate formal group of height one in the "most basic case" has been studied. The hypotheses in loc. cit. imply that the plus and minus Selmer groups over the $\Zp$-extension are finite. In particular, this does not apply to the "indefinite" anticyclotomic setting since these Selmer groups are known to be of corank one over $\Lambda$, thanks to the work of Longo--Vigni \cite{longovigni}. However, the work of Iovita--Pollack can apply to  the "definite" anticyclotomic setting. 

If we compare the formula \eqref{eq:kur} with the result of Matar (which says that $\sha(E/K_m)$ is trivial for all $m$), it suggests that in the anticyclotomic "indefinite" setting, we should not expect the terms appearing in \eqref{eq:kur} to play a role in the growth of $\sha(E/K_m)[p^\infty]$ in general.
In the present paper, we prove that this is indeed the case. In particular, we prove the following theorem, which generalizes Matar's result removing the assumptions that $D=1$, that the Tamagawa numbers of $E$ is not divisible by $p$ and that the Heegner point on $E(K)$ is not $p$-divisible.

\begin{lthm}\label{main thm}
Let $E/\QQ$ be an elliptic curve with good supersingular reduction at $p\ge 5$. Assume further that the representation \[\rho\colon G_{\Q}\to \textup{Aut}(T_p(E))\subset \textup{GL}_2(\Z_p).\] is surjective and that $\sha(E/K_m)[p^\infty]$ is finite for all $m\ge0$. Then, there exist integers $\lambda_{E,K}^\pm$, $\mu_{E,K}^\pm$ and $\nu_{E,K}$ such that 
\begin{align*}
\ord_p\left(\left|\sha(E/K_m)[p^\infty]\right|\right)&=\sum_{k\le m, k \textup{ even }}\mu_{E,K}^+ \phi(p^k)+\sum_{k\le m, k\textup{ odd }}\mu_{E,K}^-\phi(p^k)\\&+\left \lfloor \frac{m}{2}\right\rfloor\lambda_{E,K}^++\left \lfloor \frac{m+1}{2}\right\rfloor\lambda_{E,K}^-+\nu_{E,K}.\end{align*}
for all $m\gg0$.
\end{lthm}

The proof of Theorem~\ref{main thm} is divided into three steps:
\begin{itemize}
    \item[(1)] We first study in \S\ref{S:BDP} some basic properties of the so-called BDP  Selmer groups over sub-extensions of $K_\infty$ and prove a control theorem for these Selmer groups (BDP stands for Bertolini--Darmon--Prasanna since these Selmer groups are related to the $p$-adic $L$-function constructed in \cite{bertolinidarmonprasanna13}). In particular, we show that the size of $\Sel^\BDP(E/K_m)$ quotient out by the maximal divisible subgroup (which we will denote by $\Sel^\BDP(E/K_m)/\div$) satisfies a similar formula to the one given by  \eqref{eq:sha-ord}, see in particular Corollary~\ref{cor:BDP-growth}. (The definition of these Selmer groups will be reviewed in \S\ref{S:Sel}.) 
    \item[(2)] In \S\ref{S:JSW}, we show that the size of $\sha(E/K_m)[p^\infty]$ differs from the size of $\Sel^\BDP(E/K_m)/\div$ (as studied in the previous step) by the size of the kernels of certain localization maps.  A similar result was first proved by Jetchev--Skinner--Wan in \cite{jsw} and Agboola--Castella in \cite{agboolacastellaanti} under more restrictive hypotheses.  See  Theorem~\ref{BDP sha} for the precise statement.
    \item[(3)] In \S\ref{S:kernel}, we study the growth of the kernels that appear in the previous step  (see Corollary~\ref{cor:final} for the precise statement). We do so in two different ways. The first is to make use of the modules of plus and minus Heegner points introduced by Longo--Vigni in \cite{longovigni}  and Kobayashi ranks introduced in \cite{kobayashi03} (see Definitions~\ref{def:pmHeeg} and \ref{def:kob} where these concepts are reviewed). The second is to decompose the kernels in "plus" and "minus" parts and show that they satisfy certain control theorems, using ideas of Lee \cite{lee20}.
\end{itemize}

\begin{remark}
In \eqref{eq:kob}, the constants $\lambda_E^\pm$ and $\mu_E^\pm$ are the Iwasawa invariants of Kobayashi's plus and minus Selmer groups. As mentioned earlier, such Selmer groups in the "indefinite" anticyclotomic setting have been studied by Longo--Vigni in \cite{longovigni}. Unlike the cyclotomic counterparts, they are  of corank one over $\Lambda$. The constants $\lambda_{E,K}^\pm$ and $\mu_{E,K}^\pm$ in the statement of Theorem~\ref{main thm} are not directly related to the aforementioned plus and minus Selmer groups. Instead, they are given by the Iwasawa invariants of $\Sel^\BDP(E/K_\infty)^\vee$ and certain torsion $\Lambda$-modules (which arise from the plus and minus Heegner points introduced in \cite{longovigni}, or alternatively, given by the kernels of certain "plus" and "minus" localization maps; see Remarks~\ref{rk:final} and \ref{rk:final2} for details).
\end{remark}

\subsection*{Acknowledgement}
We thank Mirela Çiperiani, Ralph Greenberg, Anwesh Ray and Stefano Vigni for interesting discussions on subjects related to this article. We also thank the anonymous referee for many helpful comments and suggestions. The authors' research is supported by the NSERC Discovery Grants Program RGPIN-2020-04259 and RGPAS-2020-00096 (Lei and Müller) and the National Natural Science Foundation of China under Grant No. 11771164 (Lim). 
\section{Notation}\label{sec:notation}
We denote the absolute Galois group of $\Q$ by $G_{\Q}$ and the $p$-adic Tate module of $E$ by $T_p(E)$. Then we have a natural representation
\[\rho\colon G_{\Q}\to \textup{Aut}(T_p(E))\subset \textup{GL}_2(\Z_p).\]
We assume  that $\rho$ is surjective, in order to apply results in \cite{longovigni}.

Throughout this article, $K$ is a fixed imaginary quadratic field and $K_\infty/K$ is the anticyclotomic $\Zp$-extension. The sub-extension of degree $p^m$ is denoted by $K_m$. We assume that  $\sha(E/K_m)[p^\infty]$ is finite for all $m$.
Recall from the introduction that the conductor of $E$ is written as $N=MD$, where $D$ is a square-free product of an even number of primes and that the primes dividing $pM$ are assumed to be split in $K$, whereas those dividing $D$ are assumed to be inert in $K$. We assume that both primes above $p$ are totally ramified in $K_\infty/K$.

Let $v$ and $\overline{v}$ be the primes of $K$ lying above $p$. We fix an embedding $\iota_p:\overline\QQ\hookrightarrow \Cp$, and let $v$ be the prime induced by $\iota_p$. For $\p\in\{v,\overline{v}\}$, we write $\p_m$ for the unique prime of $K_m$ lying above $\p$. When confusion does not arise, we may omit $m$ from the notation and simply write $\p$.

We recall the following definition of "Kobayashi ranks" from \cite[\S10]{kobayashi03}.

\begin{defn}\label{def:kob}
Let $(M_n)_{n\ge1}$ be a projective system of finitely generated $\Zp$-modules with the connecting maps $\pi_n:M_n\rightarrow M_{n-1}$. If the kernel and the cokernel of $\pi_n$ are both finite, we define
    \[
     \nabla M_n=\length_{\Zp}\ker\pi_n-\length_{\Zp}\coker\pi_n+\dim_{\Qp}M_{n-1}\otimes_{\Zp}\Qp.
    \]
\end{defn}

Throughout this article, $\phi$ denotes the Euler totient function. We recall the following lemma of Kobayashi:
\begin{lemma}\label{lem:kob-rank}
Let $M$ be a finitely generated torsion $\Lambda$-module. Then, for $n\gg0$, $\nabla M_{\Gamma_n}$ is defined and is equal to $\phi(p^n)\mu(M)+\lambda(M)$, where $\mu(M)$ and $\lambda(M)$ denote the Iwasawa $\mu$-invariant and $\lambda$-invariant of the $\Lambda$-module $M$.
\end{lemma}
\begin{proof}
This is \cite[Lemma~10.5]{kobayashi03}.
\end{proof}

\begin{defn}
For a given cofinitely generated $\Zp$-module $M$, we shall write $M_{\div}$ for its maximal $p$-divisible subgroup. We then denote by $M/\div$ for $M/M_{\div}$, the quotient of $M$ by its maximal $p$-divisible subgroup.
\end{defn}

\begin{remark}\label{rk:kob}
Suppose that $\{M_n\}_{n\ge1}$ is a direct limit of cofinitely generated $\Zp$-modules such that the $\Zp$-coranks of $M_n$ stabilises  and that the connecting map $M_{n-1}\rightarrow M_n$ is injective for $n\gg0$. Then for sufficiently large $n$, the Kobayashi rank $\nabla M_n^\vee$ is defined. Furthermore,
\[
\nabla M_n^\vee=\length_{\Zp}(M_{n}/\div)-\length_{\Zp}(M_{n-1}/\div)+\rank_{\Zp}M_n^\vee.
\]
\end{remark}

We will fix a topological generator $\tau$ of $\Gamma=\Gal(K_\infty/K)$ and set $T=\tau-1$. Then the completed group ring $\Z_p\lb\Gamma\rb$ is isomorphic to $\Lambda=\Z_p\lb T\rb$. For an integer $m\ge0$, we shall write $G_m$ for the Galois group $\Gal(K_m/K)$. Furthermore, let $\Phi_m\in\Lambda$ be the $p^m$-th cyclotomic polynomial in $1+T$ and $\omega_m=(1+T)^{p^m}-1\in\Lambda$ (so $\Phi_m=\omega_m/\omega_{m-1}$ when $m\ge1$). Furthermore, we define
\[\omega_m^+(T)=\prod_{0\le n\le m, n \textup{ even}}\Phi_n(T+1),\quad \omega_m^-(T)=T\prod_{1\le n\le m, n \textup{ odd}}\Phi_n(T+1).\]

\section{Selmer groups}
\label{S:Sel}
\subsection{Plus and Minus points on the formal group}\label{Sec:pm}
 In this section, we write  $k_{m}$ for the completion of $K_m$ at either $v$ or $\overline{v}$. We denote by $\widehat{E}$ the formal group of $E$ with respect to the prime $v$ (and $\overline{v}$). By our running assumptions, $p$ splits completely in $K/\Q$ and both primes of $K$ above $p$ are totally ramified in $K_\infty/K$. Therefore, we may invoke the results of Iovita-Pollack \cite[section 4]{iovitapollack06} to see that there exist elements $d_m\in \widehat{E}(k_m)$ such that
\begin{itemize}
    \item[(1)] $\textup{Tr}_{m,m-1}(d_m)=d_{m-2}$;
    \item[(2)] $\textup{Tr}_{1,0}(d_1)=u d_0$ for a unit $u\in \Z_p^{\times}$;
    \item[(3)] $d_0$ generates $\widehat{E}(\Q_p)=\widehat{E}(k_0)$;
    \item[(4)] For every $m\ge 1$ we have an exact sequence
    \[0\to \widehat{E}(\Q_p)\to \Lambda_m d_m\oplus \Lambda_{m-1}d_{m-1}\to \widehat{E}(k_m)\to 0,\]
    where $\Lambda_n=\Lambda/\omega_n$.
\end{itemize}
We set
\[d_m^+=\begin{cases}
d_m & m \textup{ even},\\
d_{m-1} &m \textup{ odd};
\end{cases}\qquad d_m^-=\begin{cases}
d_{m-1} & m \textup{ even},\\
d_m & m \textup{ odd}.
\end{cases}.\]
We define the plus and minus subrgoup on $\widehat{E}(k_m)$ by 
\[
    \widehat{E}(k_m)^+=\left\{ x\in \widehat{E}(k_m)\mid \Tr_{m,n}(x)\in \widehat{E}(k_{n-1})\textup{ for odd $n$ with $1\le n\le m$} \right \}
\]
and 
\[
    \widehat{E}(k_m)^-=\left\{ x\in \widehat{E}(k_m)\mid \Tr_{m,n}(x)\in \widehat{E}(k_{n-1})\textup{ for even $n$ with $1\le n  \le m$} \right \}
\]

Note that $\widehat{E}(k_m)\otimes \Q_p/\Z_p\cong E(k_m)\otimes \Q_p/\Z_p$ and that the natural homomorphism $\widehat{E}^\pm(k_m)\otimes \Q_p/\Z_p\to E(k_m)\otimes \Q_p/\Z_p$ is injective  by \cite[Lemma 3.3 and remark 3.5]{longovigni}. In turn, this allows us to identify $\widehat{E}^\pm(k_m)\otimes \Q_p/\Z_p$ as a subgroup of $H^1(k_m,\Ep)$ via the Kummer map.
\begin{defn} 
For $\p\in\{v,\overline{v}\}$, we define $H^1_\pm(K_{m,\p},\Ep)$ to be the image of $\widehat{E}^\pm(k_{n})\otimes \Q_p/\Z_p$ inside $H^1(k_m,\Ep)$ under the Kummer map.

We define $H^1_\pm(K_{m,\p},\Tp)$ to be the orthogonal complement of the subgroup $H^1_\pm(K_{m,\p},\Ep)$ of  $H^1(K_{m,\p},\Ep)$\textbf{} under the local Tate pairing
\[
H^1(K_{m,\p},\Tp)\times H^1(K_{m,\p},\Ep)\rightarrow\Qp/\Zp.
\]
\end{defn}
\subsection{Defining Selmer groups}

In this section, we define the various Selmer groups that will be employed in later sections. We first define the local conditions that will be utilized.
\begin{defn}
If $w$ is a place of $K_m$, where $0\le m\le \infty$, and $W\in\{\Ep,\Tp\}$, we define
$H^1_{\Gr}(K_{m,w},W)\subset H^1(K_{m,w},W)$ to be the image of the Kummer map.

For $\p\in\{v,\overline{v}\}$, $\cL\in\{+,-,0,\emptyset\}$, we define 
\[
   H^1_{\cL}(K_{m,\p},W)=\begin{cases}
   H^1_{\pm}(K_{m,\p},W)&\text{if }\cL=\pm,\\
   0&\text{if }\cL=0,\\
   H^1(K_{m,\p},W)&\text{if }\cL=\emptyset.
   \end{cases}\]
\end{defn}

 We note that for $\cL=\emptyset$, the local condition is conventionally taken to be $H^1(K_{m,\p},W)_{\mathrm{div}}$. In our current setting, $E$ is assumed to have good supersingular reduction at  $p \geq 5$, so $E(K_{m,\p})[p^{\infty}]=0$ (by \cite[Proposition~8.7]{kobayashi03}), which in turn implies the equality $H^1(K_{m,\p},W)_{\mathrm{div}} =H^1(K_{m,\p},W)$.

These local conditions allow us to define the following Selmer groups:

\begin{defn}\label{def:Sel}
Let $W\in\{\Tp,E[p^\infty]\}$ and $0\le m\le\infty$.

   For $\cL_v,\cL_{\overline{v}}\in\{+,-,0,\emptyset,\Gr\}$, we define $\Sel^{\cL_v,\cL_{\overline{v}}}(K_m,W)$ to be the kernel of 
\[
H^1(G_\Sigma(K_m),W)\rightarrow \prod_{w\in \Sigma(K_m),w\nmid p}\frac{H^1(K_{m,w},W)}{H^1_{\Gr}(K_{m,w},W)}\times\prod_{\p\in\{v,\overline{v}\}}\frac{H^1(K_{m,\p},W)}{H^1_{\cL_\p}(K_{m,\p},W)},
\]
where $\Sigma$ denotes the set of places in $\QQ$ consisting of the bad primes of $E$, the prime $p$ and the archimedean prime,  $G_\Sigma(K_m)$ denotes the Galois group of the maximal extension of $K_m$ that is unramified outside $\Sigma$, and $\Sigma(K_m)$ is the set of primes of $K_m$ lying above $\Sigma$.

Let $\Tac=T_p(E)\otimes_{\Zp}\Lambda^\iota$, where $\Lambda^\iota$ is the $\Lambda$-module which equals $\Lambda$ as a set with the action of $\Gamma$ twisted by the inversion $\sigma\mapsto\sigma^{-1}$.
We shall write $$\Sel^{\cL_v,\cL_{\overline{v}}}(K,\Tac)=\varprojlim_n \Sel^{\cL_v,\cL_{\overline{v}}}(K_n,T_p(E))$$

and
$$X^{\cL_v,\cL_{\overline{v}}}(E/K_\infty)=\Sel^{\cL_v,\cL_{\overline{v}}}(K_\infty,E[p^\infty])^\vee.$$
\end{defn}

To simplify notation, we shall adopt the following shortcuts:
\begin{itemize}
    \item When $\cL_v=\cL_{\overline{v}}\in\{+,-\}$, we shall simply write "$\pm$" in place of "$\cL_v,\cL_{\overline v}$".
    \item We shall write "$\BDP$" in place of  "$\emptyset,0$" (since this Selmer group is related to the $p$-adic $L$-function studied by  Bertolini--Darmon--Prasanna in \cite{bertolinidarmonprasanna13}).
    \item We shall write $\Sel^{\cL_v,\cL_{\overline{v}}}(E/K_n)$ in place of $\Sel^{\cL_v,\cL_{\overline{v}}}(K_n,\Ep)$.
    \item When $\cL_v=\cL_{\overline{v}}=\Gr$, we shall remove the superscripts from the notation altogether.
\end{itemize} 
\begin{remark}
It is well known that the classical $p$-primary Selmer group $\Sel(E/K_m)$ does not depend on the definition of $\Sigma$ in the following sense (see \cite[Corollary 6.6]{Milne}): If we choose any set $\Sigma'$ containing $\Sigma$, then the Selmer group defined with respect to $\Sigma$ and the one defined with respect to $\Sigma'$ are isomorphic. A similar argument shows that the same holds for the other Selmer groups introduced in Definition~\ref{def:Sel}.
\end{remark}

\section{Properties of BDP Selmer groups}\label{S:BDP}

We first prove a control theorem for the BDP Selmer groups over $K_m$ as $m$ varies. This allows us to study the size of $\Sel^\BDP(E/K_m)$ as $m$ varies.
\begin{theorem}\label{thm:control}
The natural homomorphism
\[r_m\colon \Sel^{\BDP}(E/K_m)\to \Sel^\BDP(E/K_\infty)^{\Gamma_m}\]
is injective and has uniformly bounded cokernel. 
\end{theorem}
\begin{proof}
Consider the commutative diagram
\[\begin{tikzcd}[scale cd =0.75]
0\arrow[r]&\Sel^\BDP(E/K_m)\arrow[r]\arrow[d,"r_m"]&H^1(G_{\Sigma}(K_m),E[p^\infty])\arrow[r]\arrow[d,"h_m"]& \displaystyle\bigoplus_{w\in \Sigma(K_m)\cup\{\overline{v}\}}H^1({K_{m,w}},E[p^\infty])\arrow[d,"g_m"] \\0 \arrow[r]&\Sel^\BDP(E/K_\infty)^{\Gamma_m}\arrow[r]&H^1(G_{\Sigma}(K_\infty),E[p^{\infty}])^{\Gamma_m}\arrow[r]&  \left(\displaystyle\bigoplus_{w\in \Sigma(K_\infty)\cup\{\overline{v}\}}H^1({K_{\infty,w}},E[p^{\infty}])\right)^{\Gamma_m}\end{tikzcd}
\]

Recall that $E(K_\infty)[p^\infty]=\{0\}$ by \cite[Proposition 8.7]{kobayashi03}. Therefore, $H^i(K_\infty/K,E(K_\infty)[p^\infty])=\{0\}$ for $i=1, 2$. By the inflation restriction exact sequence, $H^i(K_\infty/K,E(K_\infty)[p^\infty])$ are the kernel and cokernel of $h_m$, respectively. Applying the snake lemma to the above diagram we see that $r_m$ is injective.

To show that the cokernel of $r_m$ is uniformly bounded, it suffices to show that the kernel of $g_m$ is uniformly bounded. Note that $g_m=\displaystyle\bigoplus_{w\in \Sigma(K_m)\cup\{\overline{v}\}}g_{m,w}$. We will bound the kernel for each prime separately.
\begin{itemize}
    \item[a)] $w=\overline{v}$. We assumed that $K_\infty/K$ is totally ramified at both primes above $p$. By the inflation restriction exact sequence, the kernel of $g_{m,\overline{v}}$ is given by $H^1(k_\infty/k_m,E(k_\infty)[p^\infty])$. As $E(k_\infty)[p^\infty]$ is trivial by \cite[Proposition 8.7]{kobayashi03}, we deduce that $g_{m,w}$ is injective.
    \item[b)] $w$ is totally split in $K_\infty/K$. In this case, $g_{m,w}$ is a diagonal embedding with trivial kernel.
    \item[c)] $w$ has a non-trivial decomposition group $G_w$ in $\Gal(K_\infty/K_m)$.  In this case, the kernel of $g_{m,w}$ is given by $H^1(G_w,E(K_{\infty,v})[p^\infty])$. The size of this group is uniformly bounded independent of $m$ by \cite[Lemma 3.4]{Lim-control}.
\end{itemize}
Note that in case c), $w$ lies above a prime $w'\in \Sigma(K)$ that is not totally split in $K_\infty/K$. As $\Sigma(K)$ is finite, the set of such $w$ is uniformly bounded independent of $m$.
\end{proof}

We prove the torsionness of the $\Lambda$-module $X^\BDP(E/K_\infty)$ with the aid of Heegner points.
\begin{defn}
For each $m\ge0$, we write $z_m\in E(K_m)$ for the Heegner point defined as in \cite[\S2.5]{BD96}.
\end{defn}

We recall from \cite[(11)]{longovigni} that these Heegner points satisfy  the following norm relations
\begin{align}
\label{norm1}    \textup{Tr}_{m,m-1}(z_m)&=-z_{m-2}\text{ for }m\ge 2,\\
 \label{norm2}   \textup{Tr}_{1,0}(z_1)&=\frac{p-1}{2}z_0,
\end{align}
     where $\Tr_{m,m-1}$ denotes the trace map from $E(K_m)$ to $E(K_{m-1})$.

\begin{defn}\label{def:pmHeeg}
For an integer $m\ge0$,  we define the plus and minus Heegner points by
\[z_m^+=\begin{cases}
z_m & m \textup{ even},\\
z_{m-1} &m \textup{ odd};
\end{cases}\qquad z_m^-=\begin{cases}
z_{m-1} & m \textup{ even},\\
z_m & m \textup{ odd}.
\end{cases}\]

We define $\cE_m^\pm$ to be the $\Lambda_m^\pm$-module generated by elements of the form $z_m^\pm\otimes p^{-k}$ (with $k\in \Z_{\ge1}$) inside $\Sel^\pm(E/K_m)$.\footnote{The norm relations \eqref{norm1} and \eqref{norm2} imply that the localization of $z_m^\pm$ belongs to $ \widehat{E}(k_m)^\pm$. Thus, $z_m^\pm\otimes p^{-k}$ lands inside $\Sel^\pm(E/K_m)$; see also \cite[Proposition~4.2]{longovigni}.}

Finally, we define $\cE_\infty^\pm$ to be the $\Lambda$-module given by $\displaystyle\varinjlim_m \cE_m^\pm$ (see \cite[Proposition~4.3]{longovigni}, which shows that the modules $\cE_m^\pm$ form a direct system).
\end{defn}

\begin{theorem}\label{thm:torsion-BDP}
The Pontryagin dual of the BDP Selmer group $X^{\BDP}(E/K_\infty)$ is a finitely generated torsion $\Lambda$-module.
\end{theorem}
\begin{proof}
This has essentially been  proved by Castella--Wan in  \cite{castellawan1607} (see in particular the equivalence given in Theorem~6.7). For completeness, we outline the ideas behind the proof. By global duality, we have the following exact sequence:
\begin{equation}\label{eq:PT1}
   \Sel^\pm(K,\Tac)\rightarrow H^1_\pm(K_v,\Tac) \rightarrow X^{\emptyset,\pm}(E/K_\infty)\rightarrow X^\pm(E/K_\infty)\rightarrow0,
\end{equation}
where $H^1_\pm(K_v,\Tac)$ is the inverse limit $\varprojlim_m H^1_\pm(K_{n,v},T)$.

By \cite[Proposition~4.16]{iovitapollack06}, $H^1_\pm(K_v,\Tac)$ is free of rank-one over $\Lambda$. Furthermore, \cite[Lemma~4.7]{castellawan1607}  tells us that $z_\infty^\pm :=(z_m^\pm)_{m\ge0}$ gives an element in the inverse limit $\Sel^\pm(K,\Tac)$. 
Furthermore, \cite[Theorem~6.2]{castellawan1607} says that the map $\mathrm{Log}_v^\pm $ (given by Definition 3.13 of op. cit) sends the localization of $z_\infty^\pm $  to a non-zero multiple of  the BDP $p$-adic $L$-function, which is non-zero (see \cite[Corollary~4.5 and Remark~4.6]{BCK}; the hypothesis of loc. cit. that the residual representation is absolutely irreducible is satisfied since we have assumed that $E$ is supersingular at $p$).  Recall from \cite[Theorem~5.1]{longovigni} that $X^\pm(E/K_\infty)$ is of rank one over $\Lambda$. Therefore, \eqref{eq:PT1} tells us that $X^{\emptyset,\pm}(E/K_\infty)$ is also of rank one over $\Lambda$.

The complex conjugation sends $X^{\emptyset,\pm}(E/K_\infty)$ to $X^{\pm,\emptyset}(E/K_\infty)$ (and vice versa). Therefore, the latter is equally of rank one over $\Lambda$. Consequently, \cite[Lemma~6.6(2)]{castellawan1607} tells us that $X^{\pm,0}(E/K_\infty)$ is a finitely generated torsion $\Lambda$-module. Since $X^{\BDP}(E/K_\infty)$  is a quotient of $X^{\pm,0}(E/K_\infty)$, the result follows.
\end{proof}

\begin{corollary}\label{cor:trivial}
The $\Lambda$-module $X^\BDP(E/K_\infty)$ does not contain any non-trivial finite submodules.
\end{corollary}
\begin{proof}
The proof of \cite[Lemma~3.4]{LMX} carries over to our current setting.
\end{proof}

\begin{remark}\label{rk:mu}
It has been proved by Hsieh \cite[Theorem B]{hsiehnonvanishing} and Burungale \cite[Theorem B]{burungale2} that the $\mu$-invariant of the BDP $p$-adic $L$-function of Bertolini–Darmon-Prasanna defined in \cite{bertolinidarmonprasanna13} (and its extension given in \cite{hunterbrooks,CastellaHsiehGHC,BCK})  is zero. Furthermore, the main conjecture predicts that the square of the BDP $p$-adic $L$-function generates the characteristic ideal of $X^\BDP(E/K_\infty)$ after base changing to the ring of integers of the completion of the maximal unramified extension of $\Qp$ (see  \cite[Theorem~1.5]{kobayashiota}, \cite[Theorem~A]{lei-zhao}, \cite[Theorem~5.3]{castellawan1607}, \cite[Theorem~8.2.1]{CLW} and \cite{CHKLL} for some recent progress towards this main conjecture in the non-ordinary setting). Therefore, it is expected that the $\mu$-invariant of $X^{\BDP}(E/K_\infty)$  should also be zero. 
\end{remark}

\begin{lemma}\label{lem:BDP-Kob}
For $m\gg0$, $\nabla(\Sel^\BDP(E/K_m)^\vee)$ is defined and is equal to 
\[
\mu_0\phi(p^m)+\lambda_0,
\]
where $\mu_0$ and $\lambda_0$ are the $\mu$- and $\lambda$-invariants of  $X^{\BDP}(E/K_\infty)$.
\end{lemma}
\begin{proof}
It follows from combining Lemma~\ref{lem:kob-rank}  with Theorems~\ref{thm:control} and \ref{thm:torsion-BDP}.
\end{proof}

\begin{corollary}\label{cor:BDP-growth}
There exists an integer $\nu_0$ such that 
\[
\ord_p\left(\left|\Sel^\BDP(E/K_m)/\div\right|\right)=\mu_0 p^m+(\lambda_0-\lambda') m+\nu_0
\]
for $m\gg0$, where $\lambda'$ is given by $$\lim_{n\rightarrow\infty}\rank_{\Zp}\Sel^\BDP(E/K_n)^\vee.$$
\end{corollary}
\begin{proof}
By Remark~\ref{rk:kob}, there exists an integer $n_0$ such that 
\[\ord_p(\left(\left|\Sel^\BDP(E/K_m)/\div\right|\right)=\sum_{n\ge n_0} \nabla \Sel^\BDP(E/K_n)^\vee-(n-n_0)\lambda'+\nu\] for all $n\ge n_0$ (note that  $\lambda'$ makes sense thanks to Theorems~\ref{thm:control} and \ref{thm:torsion-BDP}). The claim now follows from Lemma~\ref{lem:BDP-Kob}.
\end{proof}

\section{Links between Tate--Shafarevich groups and BDP Selmer groups}\label{S:JSW}

The goal of this section is to generalize a result of Jetchev--Skinnner--Wan \cite[Proposition~3.2.1]{jsw} and Agboola--Castella \cite[\S6.3]{agboolacastellaanti} on a link between the Tate--Shafarevich group and the BDP Selmer group over $K$ to  $K_m$ for all integers $m\ge0$.

\begin{defn}
For an integer $m\ge0$ and $\p\in\{v,\overline{v}\}$, we write $\psi_{m,\p}$ for the localization map
\[
E(K_m)\otimes\Qp/\Zp\rightarrow E(k_m)\otimes\Qp/\Zp,
\]
where $k_m=K_{m,\p}$. 

Furthermore, we write
   $\delta_{m,\p} = |\ker \psi_{m,\p}/\div|$.
   
When $\p$ does not play a role, we shall omit it from the notation.
\end{defn}

\begin{remark}
Since $E$ is defined over $\QQ$, it follows that $\ker \psi_{m,v}\cong \ker \psi_{m,\overline{v}}$. In particular, $\delta_{m,v}=\delta_{m,\overline{v}}$.
\end{remark}

\begin{lemma}
\label{lemma:psi-surj}
 For all $m\ge0$, the morphism $\psi_m$ is surjective.
\end{lemma}
\begin{proof}
Since $E(k_m)$ has no $p$-torsion, it is enough to show that the localization map $$E(K_m)\otimes \Qp\rightarrow E(k_m)\otimes \Qp$$ is surjective. As a $G_m$-representation, the right-hand side is isomorphic to $\Qp[G_m]$. Therefore, in order to show that the aforementioned morphism is surjective, it suffices to show that the $\chi$-component of $E(K_m)\otimes \Qp$ is non-trivial for all characters $\chi$ of $G_m$. Under our running assumption  that $\sha(E/K_m)[p^\infty]$ is finite, this is equivalent to showing that $\Sel(E/K_m)[\Phi_n]$ is infinite for all $\Phi_n$ dividing $\omega_m$.

Recall from \cite[Theorem 1.4]{longovigni} that $\Sel^\pm(E/K_\infty)^\vee$ is of rank one over $\Lambda$. The control theorem of \cite[Theorem 6.8]{iovitapollack06} tells us that $\Sel^\pm(E/K_m)[\Phi_n]$ is infinite where $n$ is chosen so that $\Phi_n|\omega_m^\pm$. Since $\Sel^\pm(E/K_m)$ is a subgroup of $\Sel(E/K_m)$, the result follows.
\end{proof}

 We  now state the main result of this section.

\begin{theorem} \label{BDP sha}
For every integer $m\ge0$, we have
\[ \big|\Sel^{\BDP}(E/K_m)/\div \big| = \big|\sha(E/K_m)[p^\infty]\big| \delta_{m,v} \delta_{m,\bar{v}}. \]
\end{theorem}

We first relate $\Sel^{\Gr,0}(E/K_m)$ with $\Sel(E/K_m)$ and $\sha(E/K_m)[p^\infty]$.

\begin{proposition} \label{v sha}
We have short exact sequences
\[ 0\longrightarrow\Sel^{\Gr,0}(E/K_m) \longrightarrow
\Sel(E/K_m) \longrightarrow E(K_{m,\overline{v}})\otimes\Qp/\Zp \longrightarrow 0,\]
\[ 0\longrightarrow \ker \psi_{m,\overline{v}} \longrightarrow  \Sel^{\Gr,0}(E/K_m) \longrightarrow\sha(E/K_m)[p^\infty]\longrightarrow 0\]
and an equality
\[ \Sel^{\Gr,\emptyset}(K_m,\Tp)=\Sel(K_m,\Tp).\]
 Furthermore, we have 
\[\big|\Sel^{\Gr,0}(E/K_m)/\div\big| =\big|\sha(E/K_m)[p^\infty]\big| \cdot\delta_{m,\overline{v}}. \]
\end{proposition}

\begin{proof}
The final asserted formula is immediate from the second short exact sequence. Therefore, it suffices to verify the two exact sequences and the final equality of compact Selmer groups.

By global duality (see for example \cite[(SES), Page 13]{skinneraws}), we have the following exact sequence:
{\small\[ 0\longrightarrow \Sel^{\Gr,0}(E/K_m) \longrightarrow \Sel(E/K_m) \longrightarrow H^1_{\Gr}(K_{m,\overline{v}},\Ep)\]\[ \longrightarrow \Sel^{\Gr,\emptyset}(K_m,\Tp)^\vee\longrightarrow \Sel(K_m,\Tp)^\vee \longrightarrow0.\]
In view of Lemma \ref{lemma:psi-surj}, the map $\Sel(E/K_m) \longrightarrow E(K_{m,\overline{v}})\otimes\Qp/\Zp$ is surjective. This establishes the first short exact sequence and the equality of compact Selmer groups. The remaining short exact sequence is a consequence of the following diagram.
\[   \entrymodifiers={!! <0pt, .8ex>+} \SelectTips{eu}{}\xymatrix{
    0 \ar[r]^{} & \ker \psi_{m,\overline{v}} \ar[d]^{} \ar[r] &E(K_{m})\otimes\Qp/\Zp
    \ar[d]^{} \ar[r] & E(K_{m,\overline{v}})\otimes\Qp/\Zp \ar@{=}[d] \ar[r] & 0 \\
    0 \ar[r]^{} & \Sel^{\Gr,0}(E/K_m) \ar[r]^{} & \Sel(E/K_m)\ar[r] &E(K_{m,\overline{v}})\otimes\Qp/\Zp \ar[r] &0.
     } \]}
\end{proof}

The next result relates $\Sel^{\Gr,0}(E/K_m)$ to $\Sel^\BDP(E/K_m)$.

\begin{proposition} \label{v BDP}
We have
\[\big|\Sel^\BDP(E/K_m)/\div\big| =\big|\Sel^{\Gr,0}(E/K_m)/\div\big| \cdot\delta_{m,{v}}. \]
 \end{proposition}

Theorem \ref{BDP sha} is then an immediate consequence of Propositions \ref{v sha} and \ref{v BDP}. It therefore remains to prove Proposition \ref{v BDP}.

\begin{proof}[Proof of Proposition \ref{v BDP}]
 By global duality, we have the following exact sequence:
\begin{equation}\label{BDPSel}
0\longrightarrow\Sel^{\Gr,0}(E/K_m) \longrightarrow\Sel^\BDP(E/K_m) \stackrel{f}{\longrightarrow} \frac{ H^1(K_{m,{v}},\Ep)}{H^1_{\Gr}(K_{m,{v}},\Ep)}
\end{equation}
\[\stackrel{g}{\longrightarrow} \Sel^{\Gr,\emptyset}(K_m,\Tp)^\vee \longrightarrow \Sel^{0,\emptyset}(K_m,\Tp)^\vee\longrightarrow 0.\]
It remains to show that $f$ has finite image of cardinality $\delta_{m,{v}}$. Equivalently,
this amounts to showing that the dual map of $g$ has finite cokernel of cardinality $\delta_{m,{v}}$. Taking the final equality of Proposition \ref{v sha} into account, the dual of $g$ is the map
\begin{align*}
     \Sel(K_m,\Tp) = \Sel^{\Gr,\emptyset}(K_m,\Tp)\longrightarrow \left(\frac{ H^1(K_{m,{v}},\Ep)}{H^1_{\Gr}(K_{m,{v}},\Ep)}\right)^\vee\\
     \cong E(K_{m,{v}})\otimes\Zp. 
\end{align*}

Write $A, B$ and $C$ for the kernel, image and cokernel of the preceding map respectively.  
Since $p$ is a supersingular prime of $E$, $E(K_{m,{v}})\otimes\Zp$ is a free $\Zp$-module, and hence so is $B$. Furthermore, $H^1(G_S(K_m), T_pE)$ and $\Sel(K_m,\Tp)$ are torsion-free $\Zp$-modules. 
Upon tensoring by $\Qp/\Zp$ and taking the equality $\Sel(E/K_m)_{\div}=\Sel(K_m,\Tp)\otimes\Qp/\Zp$ into account, we obtain two short exact sequences
\[ 0\longrightarrow A\otimes\Qp/\Zp \longrightarrow \Sel(E/K_m)_{\div} \longrightarrow B\otimes\Qp/\Zp  \longrightarrow 0, \]
\[0\longrightarrow C[p^\infty] \longrightarrow B\otimes\Qp/\Zp \longrightarrow E(K_{m,{v}})\otimes\Qp/\Zp \longrightarrow C\otimes\Qp/\Zp  \longrightarrow 0 \]
Since we are also assuming $\sha(E/K_m)[p^\infty]$ being finite, we have $E(K_m)\otimes\Qp/\Zp=\Sel(E/K_m)_{\div}$, and so the composite
\[ \Sel(E/K_m)_{\div} \longrightarrow B\otimes\Qp/\Zp \longrightarrow E(K_{m,{v}})\otimes\Qp/\Zp\]
is surjective by Lemma \ref{lemma:psi-surj}. This in turn implies that $C\otimes \Qp/\Zp=0$, or equivalently, $C$ is finite. In particular, we have $C=C[p^\infty]$.
Now, applying the snake lemma to the following commutative diagram 
\[   \entrymodifiers={!! <0pt, .8ex>+} \SelectTips{eu}{}\xymatrix{
     &   &E(K_{m})\otimes\Qp/\Zp
    \ar[d]^{} \ar@{=}[r] & E(K_{m})\otimes\Qp/\Zp \ar[d]  \\
    0 \ar[r]^{} & C \ar[r]^{} & B \otimes\Qp/\Zp \ar[r] & E(K_{m,v})\otimes\Qp/\Zp
     } \]
we obtain a short exact sequence
\[ 0 \longrightarrow A\otimes\Qp/\Zp \longrightarrow \ker \psi_{m,v} \longrightarrow C \longrightarrow 0. \]
As $C$ is finite, it follows from the exact sequence above that $ A\otimes\Qp/\Zp = (\ker \psi_{m,v})_{\div}$,
which in turn yields the equality $\delta_{m,{v}} = |C|$.
We have thus established the required formula.
\end{proof}

We end the section with the following generalization of \c{C}iperiani's main result in \cite{ciperiani}.

\begin{corollary} \label{Sha tor}
The Pontryagin dual of $\sha(E/K_\infty)[p^\infty]$ is torsion over $\Lambda$.
\end{corollary}

\begin{proof}
Upon taking direct limit of the exact sequence (\ref{BDPSel}), we see
that $\Sel^{\Gr,0}(E/K_\infty)$ is contained in  $\Sel^{\BDP}(E/K_\infty)$. (In fact, one even has an equality, although we do not require this for our proof here.) It then follows from Theorem \ref{thm:torsion-BDP} that the Pontryagin dual of $\Sel^{\Gr,0}(E/K_\infty)$ is torsion over $\Lambda$. The conclusion of the corollary is now a consequence of this latter observation and the limit of the second exact sequence in the proof of Proposition \ref{v sha}.
\end{proof}

\section{Studying kernels of localization maps}
\label{S:kernel}
In light of Corollary~\ref{cor:BDP-growth} and Theorem~\ref{BDP sha}, in order to establish Theorem~\ref{main thm}, it remains to study $\delta_{m,v}=\delta_{m,\overline v}$.  Throughout this section, we write $\psi_m$ for either $\psi_{m,v}$ or $\psi_{m,\overline v}$, write $\delta_m$ for either $\delta_{m,v}$ or $\delta_{m,\overline v}$ and write $k_m$ for either $K_{m,v}$ or $K_{m,\overline{v}}$.

\subsection{Method 1: Plus/minus Heegner points}

The following lemma allows us to view $\left((\ker\psi_m)^\vee\right)_{m\ge0}$ as an inverse system.

\begin{lemma}\label{lem:diagram}
Consider the following commutative diagram
\[\begin{tikzcd}[scale cd=0.85]
0\arrow[r] &\ker(\psi_{m-1})\arrow[r]\arrow[d,"\phi_1"]& E(K_{m-1})\otimes \Q_p/\Z_p\arrow[r]\arrow[d,"\phi_2"]& E(k_{m-1})\otimes \Q_p/\Z_p\arrow[r]\arrow[d,"\phi_3"]&0\\
0\arrow[r] &\ker(\psi_{m})\arrow[r]& E(K_{m})\otimes \Q_p/\Z_p\arrow[r]& E(k_{m})\otimes \Q_p/\Z_p\arrow[r]&0
\end{tikzcd}\]
induced by the inclusions $E(K_{m-1})\hookrightarrow E(K_m)$ and $E(k_{m-1})\hookrightarrow E(k_m)$ (note that the exactness of the horizontal arrows follows from Lemma~\ref{lemma:psi-surj}).
All three vertical maps in the commutative diagram are injective.
\end{lemma}
\begin{proof}
This follows from the fact that $E(K_m)$ has no non-trivial $p$-torsion. 
\end{proof}
We shall study the aforementioned inverse system via the plus and minus Heegner points defined in Definition~\ref{def:pmHeeg}.
\begin{defn}
For an integer $m\ge1$, we define $$e_m=\frac{\rank E(K_m)-\rank E(K_{m-1})}{\phi(p^m)}.$$
\end{defn}

\begin{theorem}\label{thm:vatsal}
There exists an integer $m_0$ such that for all $m\ge m_0$, \begin{itemize}
    \item $z_m\notin E(K_{m-1})$ and $z_m$ is not a torsion point.
    \item $e_m=1$.
\end{itemize}
\end{theorem}
\begin{proof}
The non-torsionness of $z_m$  for $m\gg 0$ was proved by Vatsal in \cite{vatsal} under the strong Heegner hypothesis (see also  \cite[Theorem~4.23]{Kim07}). For the general case, this was proved by Cornut--Vatsal (see \cite[Remark 1.9 and Theorem 1.10]{cornutvatsal2007}.

To prove that we can choose $m_0$ large enough such that $z_m\notin E(K_{m-1})$ it suffices to show that $z_m\in E(K_{m-1})$ only for finitely many $m$. If $z_m\in E(K_{m-1})$, then
\[pz_{m}=\Tr_{m,m-1}(z_m)=-z_{m-2}.\]
The natural map $E(K_m)/p^mE(K_m)\to E(K_\infty)/p^mE(K_\infty)$ is injective. Assume that there are infinitely many even $m$ such that $z_m\in E(K_{m-1})$ (the case of odd index is similar). By \cite[Lemma 4.4]{longovigni}, we can choose an index $m_0$ such that $z_{m_0}$ is not $p^{m_0}$ divisible. Let $m_1> m_0$ be the minimal even index such that $z_{m_1}\in E(K_{m_1-1})$. Define $m_2>m_1$ with the same property and $m_i$ for $i\ge 3$ analogously. Then we see that 
\[(-1)^{m_n-m_0}z_{m_0}\in p^nE(K_{m_n})\]
yielding a contradiction  for $n\ge m_0$. 

It remains to show that we can choose $m_0$ large enough to satisfy the second condition. 
Recall from Theorem \ref{thm:torsion-BDP} that $\Sel^{\BDP}(E/K_\infty)^\vee$ is torsion over $\Lambda$. Consequently, Theorem~\ref{thm:control}  implies that the $\Zp$-corank of $\Sel^{\BDP}(E/K_m)$ stabilizes for $m\gg0$. The Kummer map gives an injection
\[
\ker(\psi_m)\hookrightarrow \Sel^{\BDP}(E/K_m).
\]
In particular, the corank of $\ker(\psi_m)$ stabilizes for $m$ large enough, thanks to the injectivity of $\phi_1$ given by Lemma~\ref{lem:diagram}. For such $m$, we have
\[
\rank E(K_m)-\rank E(K_{m-1})=\rank_{\Zp} E(k_m)-\rank_{\Zp} E(k_{m-1})=\phi(p^m).
\]
Thus, $e_m=1$ for $m\gg 0$.
\end{proof}
\begin{corollary}
\label{cor:annihilator polynomials}
For each $m$, there exist $g^\pm_m\in\Lambda$ such that $(\mathcal{E}^\pm_m)^\vee\cong \Lambda/g^\pm_m$ as $\Lambda$-modules. If $m\ge m_0+1$ and $z^\pm_m\neq z^\pm_{m-1}$, then $g^\pm_m/g^\pm_{m-1}=\Phi_m$.
\end{corollary}
\begin{proof}
By \cite[Proposition~4.5]{HLV},  $(\mathcal{E}^\pm_\infty)^\vee\cong \Lambda$. As $\phi_2$ is injective, we see that $(\mathcal{E}^\pm_m)^\vee$ is $\Lambda$-cyclic, which establishes the existence of $g^\pm_m$. Now assume that $m\ge m_0+1$ and that $z^\pm_m\neq z^\pm_{m-1}$.  As $\omega^\pm_{m-1}z^\pm_{m-1}=0$, we see that $g^\pm_{m-1}\mid \omega^\pm_{m-1}$. The property $z_m\notin E(K_{m-1})$ implies that $\Phi_m\mid g^\pm_m$. As a consequence, we obtain $\Phi_m\mid (g^\pm_m/g^\pm_{m-1})$. As $\mathcal{E}^\pm_m$ is generated by $z^\pm_m$ and $\Phi_mz_m^\pm=-z^\pm_{m-1}$, it follows that $\Phi_m=g^\pm_m/g^\pm_{m-1}$.
\end{proof}

\begin{remark}
In the proof of Theorem~\ref{thm:vatsal}, we have seen that the coranks of $\ker\psi_m$ are bounded. In particular, $\nabla\ker(\psi_m)^\vee$ is defined for $m\gg0$. Together with the injectivity of $\phi_1$, this allows us to calculate $\delta_m$ via $\nabla\ker(\psi_m)^\vee$ according to Remark~\ref{rk:kob}.
\end{remark}

For the rest of this section, we fix an integer $m_0$ satisfying the conclusion of Theorem~\ref{thm:vatsal} and another integer $m\ge m_0$. For such choice of $m$, we write $\phi_i$ for the three maps given in Lemma~\ref{lem:diagram}. In what follows, we study the cokernels of these maps via the plus and minus Heegner points defined in Definition~\ref{def:pmHeeg}. Throughout, we fix $\epsilon\in\{+,-\}$ so that $z_m^\epsilon=z_m$.

\begin{lemma}
\label{lemma:cokernels-of-phi}
There exists a commutative diagram 
\[\begin{tikzcd}
0\arrow[r] &\ker(\psi_{m-1})\cap \mathcal{E}^\epsilon_{m-1}\arrow[r]\arrow[d,"\phi'_1"]& \mathcal{E}^\epsilon_{m-1}\arrow[r]\arrow[d,"\phi'_2"]& \psi_{m-1}(\mathcal{E}^\epsilon_{m-1})\arrow[r]\arrow[d,"\phi'_3"]&0\\
0\arrow[r] &\ker(\psi_{m})\cap \mathcal{E}^\epsilon_{m}\arrow[r]& \mathcal{E}^\epsilon_{m}\arrow[r]& \psi_{m}(\mathcal{E}^\epsilon_{m})\arrow[r]&0
\end{tikzcd}\]
such that all vertical maps are injective with $\coker(\phi_i')\cong \coker(\phi_i)$ for $i=1,2,3$.
\end{lemma}
\begin{proof}
The commutativity of the diagram and the injectivity are obvious from the definitions.
Recall that $m\ge m_0$ is chosen such that $z_m^\epsilon\notin E(K_{m-1})$. Therefore, the point $z^\epsilon_m$ induces a rank jump and we have
\[\corank((E(K_{m-1})+\Lambda_m z^\epsilon_m)\otimes \Q_p/\Z_p)-\corank(E(K_{m-1})\otimes \Q_p/\Z_p)=\phi(p^{m}).\]
On the other hand, the property $e_m=1$ implies that 
\[\corank(E(K_m)\otimes \Q_p/\Z_p)-\corank(E(K_{m-1})\otimes \Q_p/\Z_p)=\phi(p^{m}).\]
Thus, we obtain a natural surjection
\[\alpha\colon\mathcal{E}^\epsilon_m\to \coker(\phi_2).\]

Note that $\coker(\phi_2)$ is annihilated by the cyclotomic polynomial  $\Phi_{m}$ and is of $\Zp$-corank $\phi(p^m)$. Thus, the kernel of $\alpha$ contains all elements of the form $\Phi_{m}z_m\otimes p^{-l}$ where $l\ge0$ is an integer. Recall from  Definition~\ref{def:pmHeeg} that $$\Phi_{m}z^\epsilon_m=-z^\epsilon_{m-1}$$ showing that $\mathcal{E}^\epsilon_{m-1}$ lies in the kernel of $\alpha$. 

We obtain a surjective map
\[\alpha'\colon \mathcal{E}^\epsilon_m/\mathcal{E}^\epsilon_{m-1}\longrightarrow \coker(\phi_2)\]
with finite kernel.
Let $p^v$ be the exponent of \[E(K_m)\otimes \Z_p/(\Lambda z_m^\epsilon+E(K_{m-1})\otimes \Z_p).\] Let $x\otimes p^{-l}+(E(K_{m-1})\otimes \Q_p/\Z_p)\in \coker(\phi_2)$. Let $y\in \Lambda z^\epsilon_m$ be such that $p^vx\equiv y\mod E(K_{m-1})\otimes \Z_p$. Note that $y$ is unique modulo $\Lambda z^\epsilon_m\cap(E(K_{m-1})\otimes \Z_p)=\Lambda z^\epsilon_{m-1}$. Clearly,
\[y\otimes p^{-l-v}\equiv x\otimes p^{-l}\mod E(K_{m-1})\otimes \Q_p/\Z_p.\]
Let now $x'\otimes p^{-l'}$ be another representatve of $x\otimes p^{-l}+(E(K_{m-1})\otimes \Q_p/\Z_p)$. We can assume that $l=l'$ and $x'=x+w+p^l z$ for some $w\in E(K_{m-1})$. As $x'\otimes p^{-l}=(x'+p^lz)\otimes p^{-l}$ we can even assume that $z=0$.  Let $y'\in \Lambda z^\epsilon_m$ be such that $p^vx'\equiv y'\mod E(K_{m-1})\otimes \Z_p$. Then
\[y'-y\equiv p^v w\equiv 0\mod  E(K_{m-1})\otimes \Z_p.\]
In particular, $y-y'\in \Lambda z^\epsilon_m\cap(E(K_{m-1})\otimes \Z_p)=\Lambda z^\epsilon_{m-1}$.

We therefore obtain a well defined map
\[\alpha''\colon \coker(\phi_2)\to \mathcal{E}^\epsilon_{m}/\mathcal{E}^\epsilon_m.\]
By construction \[\alpha'(\alpha''(x\otimes p^{-l}+(E(K_{m-1})\otimes \Q_p/\Z_p))=(y\otimes p^{-l-v}+E(K_{m-1})\otimes \Q_p/\Z_p).\]
Hence, $\alpha'\circ \alpha''$ is equal to the identity.
As $\alpha'$ has a finite kernel, $\alpha''$ has a finite cokernel. But $\mathcal{E}^\epsilon_m/\mathcal{E}^\epsilon_{m-1}$ is divisible. Thus, $\alpha''$ is in fact an isomorphism and we obtain
\[\coker(\phi'_2)=\mathcal{E}^\epsilon_m/\mathcal{E}^\epsilon_{m-1}\cong \coker(\phi_2).\]

The snake lemma gives us two exact sequences
\[0\to \coker(\phi_1)\to \coker(\phi_2)\to \coker(\phi_3)\to 0\]
and 
\[0\to \coker(\phi'_1)\to \coker(\phi'_2)\to \coker(\phi'_3)\to 0,\]
which fit into the following commutative diagram \[\begin{tikzcd}
0\arrow[r] &\coker(\phi_1')\arrow[r]\arrow[d]& \coker(\phi_2')\arrow[r,"\psi_m"]\arrow[d,"\alpha'"]& \coker(\phi_3')\arrow[r]\arrow[d]&0\\
0\arrow[r] &\coker(\phi_1)\arrow[r]&\coker(\phi_2)\arrow[r,"\psi_m"]& \coker(\phi_3)\arrow[r]&0.
\end{tikzcd}\]

We have already shown that $\coker(\phi'_2)\cong \coker(\phi_2)$ is generated by $z^\epsilon_m$ and annihilated by $\Phi_m$. Clearly,
\[\psi_m(\mathcal{E}^\epsilon_{m-1})\subset \widehat{E}^\epsilon(k_{m-1})\otimes \Q_p/\Z_p\]
Using the same argument as for the injectivity of $\alpha'$, we obtain  \[\psi_m(\coker(\phi'_2))\cong(\Lambda d_m^\epsilon\otimes \Q_p/\Z_p)/(\Lambda d^\epsilon_{m-1}\otimes \Q_p/\Z_p),\]
where $d_m^\epsilon$ and $d_{m-1}^\epsilon$ are defined as in \S\ref{Sec:pm}.

Lemma \ref{lem:diagram} gives
\[\psi_m\left(\frac{E(K_m)\otimes \Q_p/\Z_p}{E(K_{m-1})\otimes \Q_p/\Z_p}\right)\cong\frac{\widehat{E}(k_m)\otimes \Q_p/\Z_p}{\widehat{E}(k_{m-1})\otimes \Q_p/\Z_p}.\]
The modules on the right hand sides of the last two displayed isomorphisms are canonically isomorphic yielding
\[\coker(\phi'_3)=\psi_m(\coker(\phi'_2))\cong \psi_m(\coker(\phi_2))=\coker(\phi_3).\]
Finally, we can deduce  the isomorphism $\coker(\phi'_1)\cong \coker(\phi_1)$ from the snake lemma.
\end{proof}

In the next two lemmas, we show that it suffices to work over $K_\infty$.
\begin{lemma}
\label{lemma:heegner-isomor}
Let $g_m^\epsilon$ be the characteristic element defined as in Corollary~\ref{cor:annihilator polynomials}. Then, there is a natural isomorphism of $\Lambda$-modules
 \[r_m\colon \mathcal{E}^\epsilon_m/\mathcal{E}^\epsilon_{m-1}\to \mathcal{E}^\epsilon_\infty[g^\epsilon_m]/\mathcal{E}^\epsilon_\infty[g^\epsilon_{m-1}].\]
\end{lemma}
\begin{proof}
As $z_m\notin E(K_{m-1})$, we have that $\mathcal{E}^\epsilon_m/\mathcal{E}^\epsilon_{m-1}$ is a divisible $\Zp$-module of corank $\phi(p^m)$ annihilated by $\Phi_m$. 

We recall from  \cite[Proposition~4.5]{HLV} that  $(\mathcal{E}^\epsilon_\infty)^\vee\cong \Lambda$.  Taking duals we obtain a homomorphism
\[r_m^\vee \colon \Lambda/\Phi_m\longrightarrow (\mathcal{E}^\epsilon_m/\mathcal{E}^\epsilon_{m-1})^\vee.\]
As $z_m\notin E(K_{m-1})$, the right hand side is annihilated by $\Phi_m$, has $\Z_p$-rank $\phi(p^m)$ and is $\Z_p$-free. Therefore, $r^\vee_m$ is either trivial or injective. As $r_m$ is non-trivial, the same holds for $r^\vee_m$. 
It follows that $r_m$ is surjective.

Note that the kernel of $r_m$ is given by $\mathcal{E}_m^\epsilon[g^\epsilon_{m-1}]/\mathcal{E}^\epsilon_{m-1}$. By Corollary \ref{cor:annihilator polynomials}, this quotient is finite. To show that it is trivial, it suffices to show that $\mathcal{E}^\epsilon_m[g^\epsilon_{m-1}]$ is $\Zp$-divisible. Clearly, $$(\mathcal{E}^\epsilon_m[g_{m-1}^\epsilon])^\vee \cong \Lambda/(g^\epsilon_m,g^\epsilon_{m-1})=\Lambda/g^\epsilon_{m-1}.$$
By Corollary~\ref{cor:annihilator polynomials}, the latter is $\Z_p$-free giving the desired divisibility of $\mathcal{E}_m^\epsilon[g^\epsilon_{m-1}]$.
\end{proof}
\begin{remark}
Note that it is crucial here to use the polynomials $g^\epsilon_m$ instead of $\omega^\epsilon_m$. If the latter were used,  the resulting map $r_m$ would  still be surjective. But in the case $g_{m-1}^\epsilon\neq \omega^\epsilon_{m-1}$, the $\Zp$-module $\mathcal{E}^\epsilon_{m-1}[\omega^\epsilon_{m-1}]$ would not be divisible, forcing the quotient $\mathcal{E}^\epsilon_m[\omega^\epsilon_m]/\mathcal{E}^\epsilon_{m-1}$ to be non-trivial.
\end{remark}
\begin{lemma}
\label{lem:diagram-at-infinity}Consider the following commutative diagram:
\[\begin{tikzcd}[scale cd =0.9]
0\arrow[r] &\ker(\psi_\infty)\cap \mathcal{E}^\epsilon_\infty[g^\epsilon_{m-1}]\arrow[r]\arrow[d,"\phi''_1"]& \mathcal{E}^\epsilon_\infty[g^\epsilon_{m-1}]\arrow[r]\arrow[d,"\phi''_2"]& \psi_\infty(\mathcal{E}^\epsilon_\infty[g^\epsilon_{m-1}])\arrow[r]\arrow[d,"\phi''_3"]&0\\
0\arrow[r] &\ker(\psi_{\infty})\cap \mathcal{E}^\epsilon_{\infty}[g^\epsilon_{m}]\arrow[r]& \mathcal{E}^\epsilon_{\infty}[g^\epsilon_{m}]\arrow[r]& \psi_{\infty}(\mathcal{E}^\epsilon_{\infty}[g^\epsilon_{m}])\arrow[r]&0.
\end{tikzcd}\]
All vertical maps are injective and $\coker(\phi''_i)\cong \coker(\phi'_i)$ for $i=1,2,3$.
\end{lemma}
\begin{proof}
The injectivity of the vertical maps is immediate, as before. By Lemma~\ref{lemma:heegner-isomor},
\[\coker(\phi''_2)\cong \coker(\phi'_2)\]
Note that the restriction of $\psi_\infty$ to $E(K_m)\otimes \Q_p/\Z_p$ coincides with $\psi_m$. The rest of the proof is analogous to the one of Lemma~\ref{lemma:cokernels-of-phi}.
\end{proof}

The lemmas above allow us to reduce the study of $\nabla\ker(\psi_m)^\vee$ to computing the image of $\ker( \psi_\infty)\cap \mathcal{E}^\epsilon_\infty[g^\epsilon_m]$ in $\mathcal{E}^\epsilon_\infty[g^\epsilon_m]/\mathcal{E}^\epsilon_\infty[g^\epsilon_{m-1}]$. 

As stated in the proof of Lemma~\ref{lemma:heegner-isomor},   $(\mathcal{E}^\epsilon_\infty)^\vee\cong \Lambda$. This results in the short exact sequence
\[0\to \Lambda\to \Lambda\to (\ker(\psi_\infty)\cap \mathcal{E}^\epsilon_\infty)^\vee\to 0,\] from which we deduce that
 \[(\ker( \psi_\infty)\cap \mathcal{E}^\epsilon_\infty)^\vee\cong \Lambda/f^\epsilon\] for some power series $f^\epsilon\in \Lambda$.

Choose an index $m'_0$ such that $\gcd(f^\epsilon,g^\epsilon_m)=\gcd(f^\epsilon,g^\epsilon_{m'_0})$ for all $m\ge m'_0$. For the rest of the section we will assume that $m\ge \max(m_0,m'_0)$, where $m_0$ is defined as in Theorem~\ref{thm:vatsal}.

 \begin{lemma}
 \label{lem:cokernel-phi''}
 Let $\lambda^\epsilon=\deg(f^\epsilon/\gcd(f^\epsilon,g^\epsilon_{m'_0}))$ and $\mu^\epsilon=\mu(f^\epsilon)$. Then we have
 \[\ord_p\left(\left|\coker(\phi''_1)\right|\right)=\lambda^\epsilon+\mu^\epsilon \phi(p^{m})\]
 for $m \gg 0$.
 \end{lemma}
 \begin{proof}
 There is a natural map\[\theta \colon \ker(\psi_\infty)\cap \mathcal{E}^\epsilon_\infty[g^\epsilon_m]\to \coker(\phi''_2).\]
 Note that $\coker(\phi''_1)\cong \textup{Im}(\theta)$. We have a cononical isomorphism $\textup{Im}(\theta)^\vee\cong\theta^\vee(\coker(\phi''_2)^\vee)$, where
 \[\theta^\vee \colon \coker(\phi''_2)^\vee\to (\ker(\psi_\infty)\cap \mathcal{E}^\epsilon_\infty[g^\epsilon_m])^\vee\] is the dual of $\theta$.
 
 Recall that $(\mathcal{E}^\epsilon_\infty)^\vee \cong \Lambda$. Therefore \[\coker(\phi''_2)^\vee \cong g^\epsilon_{m-1}\Lambda/g^\epsilon_m\Lambda.\]
 By the definition of $f^\epsilon$, we have
 \[\theta^\vee \colon g^\epsilon_{m-1}\Lambda/g^\epsilon_m\Lambda\to \Lambda/(f^\epsilon,g^\epsilon_m).\]
 We now define \[\theta'\colon \Lambda/(f^\epsilon,g^\epsilon_m)\to \Lambda/(f^\epsilon, g^\epsilon_{m-1}).\]
 Then the image of $\theta^\vee$ is canonically isomorphic to $\ker(\theta')$. We can deduce that
 \[\ord_p(\vert \ker (\theta')\vert)=\lambda^\epsilon+\mu^\epsilon \phi(p^m)\]
 using a standard argument (see for example \cite[Section 13]{washington}).
 \end{proof}
 As a consequence, we deduce the following:
\begin{lemma}\label{lem:kob-psi}
Let $\epsilon=(-1)^m$. For $m\gg 0$ have 
\[\nabla(\ker(\psi_m))=\lambda^\epsilon+\mu^\epsilon \phi(p^m)+\corank(\ker(\psi_m)).\]
\end{lemma}
\begin{proof}
Consider the canonical homomorphism 
\[\ker(\psi_m)^\vee\to \ker(\psi_{m-1})^\vee.\]
By Lemma~\ref{lem:diagram}, this homomorphism is surjective. By Lemmas~\ref{lemma:cokernels-of-phi}, \ref{lem:diagram-at-infinity} and \ref{lem:cokernel-phi''}, the kernel has cardinality $\lambda^\epsilon+\mu^\epsilon \phi(p^m)$.\end{proof}
As a final corollary, we obtain the main result of this section:
\begin{corollary}\label{cor:final}
There is a constant $\nu$ such that for $m\gg 0$ we have
\begin{align*}
\ord_p\left(\left|\ker(\psi_m)/\div\right|\right)&=\sum_{k\le m, k \textup{ even }}\mu^+ \phi(p^k)+\sum_{k\le m, k\textup{ odd }}\mu^-\phi(p^k)\\&+\left \lfloor \frac{m}{2}\right\rfloor\lambda^++\left \lfloor \frac{m+1}{2}\right\rfloor\lambda^-+\nu.\end{align*}
\end{corollary}
\begin{proof}
This follows from combining Lemma~\ref{lem:kob-psi} with the formula presented in Remark~\ref{rk:kob}.
\end{proof}

\begin{remark}\label{rk:final}
Combining Corollary~\ref{cor:final} with Theorem~\ref{BDP sha} and Corollary~\ref{cor:BDP-growth} proves 
Theorem~\ref{main thm}. Note that we have
\[
\mu_{E,K}^\pm=\mu_0-2\mu^\pm,\quad \lambda_{E,K}^\pm=\lambda_0-\lambda'-2\lambda^\pm.
\]
Note that if $\mu_0=\lambda_0=0$, then $\Sel^\BDP(E/K_\infty)$ is trivial in view of Corollary~\ref{cor:trivial}). It then follows from Theorem~\ref{thm:control} that $\Sel^\BDP(E/K_m)$ is also trivial for all $m$. In particular, Theorem~\ref{BDP sha} implies that  $\sha(E/K_m)[p^\infty]$ is trivial for all $m$.
\end{remark}

\subsection{Method 2: Plus/minus subgroups and control theorems}
We follow \cite{agboolahowardsupersingular} to define the following plus and minus subgroups:
\begin{defn}\label{def:pm}
For an integer $ m\ge 0$, we define $\Xi_m^+$ (respectively $\Xi_m^-$) to be the set of characters of $G_m$ of exact order $p^k$ for $k$ even (respectively odd).

We define $\tEpm(K_m)$ to be the set of elements
\[\left\{
x\in E(K_m)\otimes_{\Zp}\Qp:\sum_{\sigma\in G_m}\chi(\sigma)x^\sigma=0,\ \forall\chi\in\Xi_m^\mp\right\}
\]
and define $\cM^\pm(K_m)$ to be the image of $\tEpm(K_m)$ under
\[
E(K_m)\otimes\Qp\rightarrow E(K_m)\otimes\Qp/\Zp\rightarrow H^1(K_m,\Ep),
\]
where the last map is the Kummer map. We define $\cM^\pm(K_\infty)$ to be the direct limit of $\cM^\pm(K_m)$ as $m$ runs over all non-negative integers (which results in a subgroup of $H^1(K_\infty,\Ep)$.

Similarly, for $k_m=K_{m,\p}$ where $\p\in\{v,\overline{v}\}$, we define the subgroups $\cM^\pm(k_m)\subset H^1(k_m,\Ep)$ for an integer $m\ge0$ in the same manner as $\cM^\pm(K_m)$ and  $\cM^\pm(k_\infty)=\displaystyle\varinjlim_m \cM^\pm(k_m)\subset H^1(k_\infty,\Ep)$.

Finally, for $0\le m\le\infty$, we define $\cM(K_m)$ and $\cM(k_m)$ to be the images of $E(K_m)\otimes\Qp/\Zp$ and $E(k_m)\otimes\Qp/\Zp$ inside $H^1(K_m,\Ep)$ and $H^1(k_m,\Ep)$ respectively.
\end{defn}

\begin{remark}\label{rk:diff-pm}
Note that for $0\le m\le \infty$, $\cM^\pm(k_m)$ are in fact subgroups of  $H^1_\pm(k_m,\Ep)$. Furthermore,  equality holds for the "plus" group. For the "minus" group, we have \[\cM^-(k_m)=(\tau-1)H^1_-(k_m,\Ep).\]
(Recall that $\tau$ is a topological generator of $\Gamma$.)\end{remark}

\begin{defn}
Given a $\Lambda$-module $M$, we write $M[\omega_m^\pm]$ for the kernel of the multiplication-by-$\omega_m^\pm$ map in $M$ (where $\omega_m^\pm$ is as defined at the end of \S\ref{sec:notation}).
\end{defn}

\begin{remark}
Note that the $G_m$-representation $E(K_m)\otimes \Qp$ is isomorphic to $\widetilde{E}^+(K_m)\oplus \widetilde{E}^-(K_m)$. Furthermore, $\left(E(K_m)\otimes \Qp\right)[\omega_m^\pm]$ is precisely $\tEpm(K_m)$.
\end{remark}

We prove several basic properties of the plus and minus subgroups we have defined in Definition~\ref{def:pm}. 

\begin{proposition}\label{prop:M-pm}
The following properties hold:
\begin{itemize}
     \item[(1)] For all $m\ge0$, $\cM^+(K_m)\oplus \cM^-(K_m)\cong \cM(K_m)$ as $\Lambda$-modules. The same holds if we replace $K_m$ by $k_m$.
    \item[(2)] There exist injective  $\Lambda$-morphisms $$\cM^\pm(K_\infty)^\vee\hookrightarrow \Lambda^{r^\pm}\oplus \bigoplus_{n\in I^\pm}( \Lambda/\Phi_{n})^{t_n}$$ for some integers $r^\pm,t_n\in\ZZ_{\ge0}$ and finite sets $I^\pm\subset\ZZ_{\ge0}$ containing only even/odd integers. Furthermore, the cokernels of these maps are finite.
  
    \item[(3)] The $\Lambda$-modules $\cM^\pm(k_\infty)^\vee$ are free of rank one.
\end{itemize}
\end{proposition}

\begin{proof}
As a $G_m$-Galois representation, it is clear that $E(K_m)\otimes\Qp$ decomposes as $\widetilde E^+(K_m)\oplus \widetilde E^-(K_m)$. Thus, we have a surjection
\[
\cM^+(K_m)\oplus \cM^-(K_m)\rightarrow \cM(K_m).
\]
The injectivity of this map follows from the argument presented in \cite[proof of Lemma~8.17]{kobayashi03}. It is clear that the proof equally applies if we replace $K_m$ by $k_m$. This proves the statement (1).

It follows from (1) that $\cM^\pm(K_m)$ is a cofinitely generated cofree $\Zp$-module. Furthermore, there is a natural map 
\[
\cM^\pm(K_m)\hookrightarrow \cM^\pm(K_\infty)[\omega_m^\pm].
\]Therefore, the proof of \cite[Theorem~2.1.2]{lee20} can be carried over, after replacing the functors on finitely generated $\Lambda$-modules
\[
\mathfrak{F}:M\mapsto\left(\varinjlim_m \frac{M}{\omega_m M}[p^\infty]\right)^\vee\quad\text{and}\quad  \mathfrak{G}:M\mapsto\left(\varprojlim_m \frac{M}{\omega_m M}[p^\infty]\right)
\]
by the plus and minus counterparts
\[
\mathfrak{F}^\pm:M\mapsto\left(\varinjlim_m \frac{M}{\omega_m^\pm M}[p^\infty]\right)^\vee\quad\text{and}\quad \mathfrak{G}^\pm:M\mapsto\left(\varprojlim_m \frac{M}{\omega_m^\pm M}[p^\infty]\right)
\] respectively. In particular, we obtain that \[\varprojlim_{m}\left(\frac{\cM^\pm(K_\infty)^\vee}{\omega_m^\pm\cM^\pm(K_\infty)^\vee}\right)[p^\infty]=0.\] From this, we can conclude as in \cite[Lemma A.2.9]{lee20} that $\cM^\pm(K_\infty)^\vee$ has the desired form,  proving (2).

We recall that
\[
H^1_\pm(k_\infty,\Ep)^\vee\cong \Lambda
\]
as $\Lambda$-modules by \cite[Proposition~4.16]{iovitapollack06}. Thus, the assertion now follows from Remark~\ref{rk:diff-pm}.
\end{proof}

We define the following Selmer groups:
\begin{defn}
For $0\le m\le \infty$, we define $\cS^\pm(E/K_m)$ to be the kernel of
\[
H^1(G_\Sigma(K_m),\Ep)\rightarrow \prod_{\substack{w\in \Sigma(K_m)\\w\nmid p}}\frac{H^1(K_{m,w},\Ep)}{H^1_{\Gr}(K_{m,w},\Ep)}\times\prod_{\p\in\{v,\overline{v}\}}\frac{H^1(K_{m,\p},\Ep)}{\cM^\pm(K_{m,\p})}.
\]
\end{defn}

\begin{theorem}\label{thm:controlpm}The following properties hold:
\begin{itemize}
    \item[(1)] The restriction map induces an injective $\Lambda$-morphism
\[
\cS^\pm(E/K_m)[\omega_m^\pm]\rightarrow \cS^\pm(E/K_\infty)[\omega_m^\pm]
\]
with finite cokernel whose cardinality is bounded independently of $m$. 
\item[(2)] There is a natural $\Lambda$-homomorphism
\[
 \cS^+(E/K_m)[\omega_m^+]\oplus  \cS^-(E/K_m)[\omega_m^-]\rightarrow \Sel(E/K_m)
\]
with finite kernel and cokernel.
\item[(3)] The $\Lambda$-modules $\cS^\pm(E/K_\infty)$ are of corank one.
\end{itemize}

\end{theorem}
\begin{proof}
(1) follows from the same proof as \cite[Theorem~9.3]{kobayashi03}, even though our local conditions are slightly different. See \cite[Theorem~5.2]{agboolahowardsupersingular} where a similar result has been proved under the local conditions we have chosen.

(2) follows from the same proof as \cite[Proposition~5.3]{agboolahowardsupersingular}.

Finally, Remark~\ref{rk:diff-pm} tells us that $\Lambda$-modules $\cS^\pm(E/K_\infty)^\vee$ and $X^\pm(E/K_\infty)$ have the same rank. Thus, (3) follows from \cite[Theorem~5.1]{longovigni}.
\end{proof}

\begin{theorem}\label{thm:control-M}
The restriction map induces  injective $\Lambda$-homomorphisms
\[
s_m^\pm:\cM^\pm(K_m)\rightarrow \cM^\pm(K_\infty)[\omega_m^\pm]
\]
with finite cokernel whose cardinality is bounded independently of $m$. 
\end{theorem}
\begin{proof}
The injectivity of $s_m^\pm$ is a consequence of the injectivity of the restriction map on cohomology (which in turn follows from the fact that $E(K_\infty)[p]=0$ and the inflation-restriction exact sequence). In order to show that the cokernels are finite, it is enough to compare the $\Zp$-coranks of the modules on the two sides.

Let us write 
\begin{align*}
    a_m^\pm&=\corank_{\Zp}\cM^\pm(K_m),\\
    b_m^\pm&=\corank_{\Zp}\cM^\pm(K_\infty)[\omega_m^\pm],\\
    c_m^\pm&=\corank_{\Zp}\cS^\pm(E/K_m)[\omega_m^\pm]=\corank_{\Zp}\cS^\pm(E/K_\infty)[\omega_m^\pm],\\
    r_m&=\rank E(K_m).
\end{align*}
(Note that  the equality in the  definition of $c_m^\pm$ is a consequence of Theorem~\ref{thm:controlpm}(1).)
Since we have assumed that $\sha(E/K_m)[p^\infty]$ is finite,  Theorem~\ref{thm:controlpm}(2) implies that
\[
r_m=c_m^++c_m^-.
\]
Combined with Proposition~\ref{prop:M-pm}(1), we deduce that
\begin{equation}
   r_m=a_m^++a_m^-=c_m^++c_m^-.\label{eq:compare-corank}
\end{equation}
Furthermore, it follows from definitions that there are natural inclusions
\begin{align*}
    \cM^\pm(K_m)&\hookrightarrow \cS^\pm(E/K_m)[\omega_m^\pm],\\\cM^\pm(K_\infty)[\omega_m^\pm]&\hookrightarrow \cS^\pm(E/K_\infty)[\omega_m^\pm] .
\end{align*} In particular, we have the following inequalities
\begin{align}
    a_m^\pm\le c_m^\pm,\label{eq:a-c}\\ b_m^\pm\le c_m^\pm.\label{eq:b-c}
\end{align}
Thus, if we combine \eqref{eq:compare-corank} and \eqref{eq:a-c}, we deduce that $a_m^\pm=c_m^\pm$. Combined this with \eqref{eq:b-c} tells us that $b_m^\pm\le a_m^\pm$. The injectivity of $s_m^\pm$ says that $a_m^\pm\le b_m^\pm$. Therefore, $a_m^\pm=b_m^\pm$ as desired.

 It remains to show that the cardinalities of $\coker s_m^\pm$ are uniformly bounded. As in the proof of \cite[Lemma~4.2.4]{lee20}, we have
 \[
 (\coker s_m^\pm)^\vee\cong \left(\cM^\pm(K_\infty)^\vee/\omega_m^\pm\right)[p^\infty],
 \]
which have bounded cardinalities by Proposition~\ref{prop:M-pm}(2).
\end{proof}

\begin{corollary}\label{cor:global-corank}
The integers $r^\pm$ in Proposition~\ref{prop:M-pm}(2) are both equal to one.
\end{corollary}
\begin{proof}
By Theorem~\ref{thm:controlpm}(1) and (3), the $\Zp$-coranks $c_m^\pm$ in the proof of Theorem~\ref{thm:control-M} are give by $\deg(\omega_m^\pm)+O(1)$, which is also equal to $\Zp$-coranks of $ \cM^\pm(K_\infty)[\omega_m^\pm]$. Therefore, the $\Lambda$-coranks of $\cM^\pm(K_\infty)$ are equal to one.
\end{proof}
We have a somewhat more simple version of Theorem~\ref{thm:control-M} for the local plus and minus groups.
\begin{lemma}\label{lem:control-local}
 The $\Lambda$-modules $\cM^\pm(k_m)$ and $\cM^\pm(k_\infty)[\omega_m^\pm]$ are isomorphic.
\end{lemma}
\begin{proof}
As before, the restriction in cohomology induces an injection $$\cM^\pm(k_m)\hookrightarrow \cM^\pm(k_\infty)[\omega_m^\pm].$$
Proposition~\ref{prop:M-pm}(3) tells us that $\cM^\pm(k_\infty)[\omega_m^\pm]$ is a cofree $\Zp$-module of corank $\deg(\omega_m^\pm)$. The same is true for $\cM^\pm(k_m)$ by considering the decomposition given by Proposition~\ref{prop:M-pm}(1) and the isomorphism of $G_m$-modules $\cM(k_m)\cong \Qp/\Zp [G_m]$. Therefore, the injection above is also surjective.
\end{proof}

We now define the plus and minus versions of the localization map $\psi_m$.
\begin{defn}
For $0\le m\le \infty$, we define $\psi_m^\pm:\cM^\pm(K_m)\rightarrow \cM^\pm(k_m)$ to be the morphisms induced by the localization map $E(K_m)\rightarrow E(k_m)$.
\end{defn}
\begin{lemma}\label{lem:ker-cotor}
 The maps $\psi_\infty^\pm$ is surjective and $\ker(\psi_\infty^\pm)$ are $\Lambda$-cotorsion.
\end{lemma}
\begin{proof}
The surjectivity follows from Lemma~\ref{lemma:psi-surj} and the decompositions given in Proposition~\ref{prop:M-pm}(1). Furthermore, Proposition~\ref{prop:M-pm}(3) and Corollary~\ref{cor:global-corank} tell us that $\cM^\pm(K_\infty)$ and $\cM^\pm(k_\infty)$ have the same $\Lambda$-corank, hence the kernels are cotorsion over $\Lambda$ as claimed.
\end{proof}

By Proposition~\ref{prop:M-pm}(1), we have
\[
\ker(\psi_m)\cong \ker(\psi_m^+)\oplus \ker(\psi_m^-). 
\]
It follows from Theorem~\ref{thm:control-M} and Lemma~\ref{lem:control-local} that there are injections
\[
\ker(\psi_m^\pm)\hookrightarrow \ker(\psi_\infty^\pm)[\omega_m^\pm]
\]
with finite cokernels which are bounded uniformly in $m$. Lemma~\ref{lem:ker-cotor} tells us that $\ker(\psi_\infty^\pm)$ are $\Lambda$-cotorsion. From here, we can  estimate the growth of $\delta_m$ as in Method 1.

\begin{remark}\label{rk:final2}
We see that the constants $\lambda^\pm$ and $\mu^\pm$ in Corollary~\ref{cor:final} can be realized as the corresponding Iwasawa invariants of $\ker(\psi_\infty^\pm)^\vee$. In particular, the formulae of the constants $\mu_{E,K}^\pm$ and $\lambda_{E,K}^\pm$ given in Remark~\ref{rk:final} may be rewritten as
\[
\mu_{E,K}^\pm=\mu_0-2\mu(\ker(\psi_\infty^\pm)^\vee),\quad \lambda_{E,K}^\pm=\lambda_0-\lambda'-2\lambda(\ker(\psi_\infty^\pm)^\vee).
\]
\end{remark}
\bibliographystyle{amsalpha}
\bibliography{references}

\end{document}